\documentclass[11pt]{amsart}
\usepackage{amscd,amssymb,float}
\input xy
\xyoption{all}

\newtheorem{theorem}{Theorem}[section]
\newtheorem{lemma}[theorem]{Lemma}

\newtheorem{proposition}[theorem]{Proposition}

\renewcommand{\leq}{\leqslant}
\renewcommand{\geq}{\geqslant}

\theoremstyle{definition}

\theoremstyle{definition}

\numberwithin{equation}{section}

\newcommand{\ve}{\varepsilon}

\newcommand{\ti}[1]{\tilde{#1}}

\newcommand{\CC} {\mathbb{C}}

\renewcommand{\leq}{\leqslant}
\renewcommand{\geq}{\geqslant}
\renewcommand{\le}{\leqslant}
\renewcommand{\ge}{\geqslant}

\numberwithin{equation}{section} \numberwithin{figure}{section}

\title{Gromov-Hausdorff collapsing of Calabi-Yau manifolds}
\author[M. Gross]{Mark Gross$^{*}$}
\thanks{$^{*}$Supported in part by NSF grant DMS-1105871.}
\address{Mathematics Department, University of California San Diego, 9500 Gilman Drive \#0112, La Jolla, CA 92093}
\email{mgross@math.ucsd.edu}
\author[V. Tosatti]{Valentino Tosatti$^{\dagger}$}
\thanks{$^{\dagger}$Supported in part by a Sloan Research Fellowship and NSF grant DMS-1236969.}
 \address{Department of Mathematics, Northwestern University, 2033 Sheridan Road, Evanston, IL 60201}
  \email{tosatti@math.northwestern.edu}
  \author[Y. Zhang]{Yuguang Zhang$^{\ddagger}$}
  \thanks{$^{\ddagger}$Supported in part  by NSFC-11271015.}
\address{Mathematics Department, Capital Normal University, Beijing 100048, P.R.China.}
\email{yuguangzhang76@yahoo.com}

\begin{document}
\begin{abstract}
This paper is a sequel to \cite{GTZ}.  We further study Gromov-Hausdorff collapsing limits  of Ricci-flat K\"ahler metrics on abelian fibered    Calabi-Yau manifolds. Firstly,  we show that in the same setup as \cite{GTZ},  if the dimension of the  base manifold is  one, the limit  metric space     is homeomorphic to the base manifold.  Secondly,  if the fibered Calabi-Yau manifolds are Lagrangian fibrations of holomorphic  symplectic manifolds, the metrics on the regular parts of the limits  are special K\"ahler metrics.  By combining these two results, we extend \cite{GW} to any fibered projective  K3 surface without any assumption on  the type of singular fibers.
\end{abstract}
\maketitle
\section{Introduction}

In this paper we continue our study in \cite{GTZ} of the structure of collapsed Gromov-Hausdorff limits of Ricci-flat K\"ahler metrics on compact Calabi-Yau manifolds.
Let $M$ be a projective Calabi-Yau manifold of complex dimension $m$, with $\Omega$ a nowhere vanishing holomorphic $m$-form on $M$. Let $N$ be a projective manifold of dimension $0<n<m$, and $f:M \to N$ be a holomorphic fibration (i.e., a surjective holomorphic map with connected fibers) whose general fibre
is an abelian variety.  Let $\alpha$ be an ample class on $M$, and  let $N_{0}\subset N$ be the Zariski  open subset such that, for any $y\in N_{0}$, $M_{y}=f^{-1}(y)$ is  smooth (and therefore Calabi-Yau). Let $D=N\backslash N_{0}$ be the discriminant locus of the map $f$.
Let $\alpha_{0} $  be an  ample class on $N$,  and $\tilde{\omega}_{t}\in f^{*}\alpha_{0} +t\alpha$ be the Ricci-flat K\"ahler metric given by Yau's Theorem \cite{Ya} for $t\in  ( 0, 1 ]$,  which satisfies the complex Monge-Amp\`ere equation
\begin{equation}\label{eq1}  \tilde{\omega}_{t}^{m}=c_{t} t^{m-n}(-1)^{\frac{m^{2}}{2}}\Omega  \wedge\overline{\Omega}.
\end{equation}
By \cite{To1}   and \cite{GTZ},   $\tilde{\omega}_{t}$  converges smoothly  to $f^{*}\omega$ on $f^{-1}(K)$ for any compact $K\subset N_{0}$ as $t\rightarrow 0$, and
on $N_{0}$,  where $\omega$ is the    K\"ahler metric  on $N_{0}$ with $${\rm Ric}(\omega) =\omega_{WP}$$
 obtained in \cite{To1} and \cite{ST2} (see also \cite{ST1}), and
 $\omega_{WP}$ is a Weil-Petersson semipositive form on $N_0$ coming from the variation of the complex structures of the fibers $M_y$.  Furthermore, the Ricci-flat metrics $\ti{\omega}_t$ have locally uniformly bounded curvature on $f^{-1}(N_0)$. Thanks to \cite{To0} and \cite{Z},  the diameter of the metrics $\ti{\omega}_t$ satisfies
$${\rm diam}_{\tilde{\omega}_{t}}(M) \leqslant C,$$
for a constant $C>0$ independent of $t$. Gromov's pre-compactness theorem (cf.\ \cite{Gr1}) then implies that for any sequence $t_{k} \to 0$,  a subsequence of $(M, \tilde{\omega}_{t_{k}})$ converges to a compact metric space $(X, d_{X})$ in the Gromov-Hausdorff sense (a priori different subsequences could result in non-isometric limits).  In our earlier work \cite{GTZ} we proved that $(N_{0}, \omega)$ can be locally isometrically  embedded into $(X, d_{X})$, with open dense image $X_0\subset X$, via a homeomorphism $\phi:N_0\to X_0$.
The following questions remain open (see \cite{To2, To3}).
\begin{enumerate}
\item Is $(X,d_X)$ isometric to the metric completion of $(N_{0}, \omega)$?
\item Is the real Hausdorff codimension of $X\backslash X_0$ at least $2$?
\item Is $X$ homeomorphic to $N$?
\end{enumerate}
Note that (1) implies that there is no need to pass to any subsequence to obtain Gromov-Hausdorff convergence.

The first result of this paper is that all these questions have an affirmative answer if $N$ has complex dimension $1$.

 \begin{theorem}\label{theorem1}   If  $n=1$,  then the Gromov-Hausdorff limit
$$(M, \tilde{\omega}_{t}) \stackrel{d_{GH}}\longrightarrow  (X, d_X),$$
exists as $t \to 0$, and is isometric to the metric completion of $(N_0,\omega)$. Furthermore, $X\backslash X_0=X\backslash\phi(N_0)$ is a finite number of points,
and $X$ is homeomorphic to $N$.
\end{theorem}

In the case of  non-collapsing limits,  analogous results about metric completions   have been  obtained in \cite{RZ, RZ2, RuZ}, and homeomorphism results have been obtained in \cite{DS} (see also \cite{Tian}).

Now we drop the assumption that the dimension of $N$ equals $1$, and assume instead that $M$ is irreducible (i.e., simply connected and not the product of two lower-dimensional complex manifolds) and that it admits a holomorphic symplectic form $\Theta$, which is a non-degenerate holomorphic 2-form. In this case the complex dimension of $M$ must be even, $m=2n$, and $(M, \Theta)$ is called an irreducible holomorphic symplectic manifold. If $f:M \to N$ is a fibration as before, then it is known \cite{Ma} that all smooth fibers $M_y$ are complex $n$-tori, which are holomorphic Lagrangian, in the sense that  $\Theta|_{M_{y}}=0$. Also, the base manifold $N$ is always biholomorphic to $\mathbb{CP}^n$ \cite{Hw, GL}.
Furthermore, we have the nowhere vanishing holomorphic $2n$-form $\Omega=\Theta^{n}$, and the Ricci-flat equation \eqref{eq1} says exactly that the metrics $\ti{\omega}_t$ for any $0<t\leq 1$ are all hyperk\"ahler (i.e., their Riemannian holonomy equals $Sp(n)$).

In this case, the results of our previous work \cite{GTZ} apply, and we have again that the Ricci-flat metrics $\ti{\omega}_t$ collapse smoothly with locally bounded curvature
to a K\"ahler metric $\omega$ on $N_0$, with the same properties as before. To state our next result, we need a few definitions. A K\"ahler metric $\omega$ on a complex manifold $(N,J)$ is called a {\em special K\"ahler metric} \cite{Fr} if there is a real flat torsion-free connection $\nabla$ on $M$ with $\nabla\omega=0$ and such that
$$d^{\nabla}J=0,$$
where $d^{\nabla}:\Omega^p(TN)\to \Omega^{p+1}(TN)$ is the extended deRham complex ($(d^{\nabla})^2=0$ since $\nabla$ is flat), and we are viewing $J$ as an element of $\Omega^1(TN)$. This notion originated in the physics literature, and has been extensively studied, see e.g. \cite{Co, Fr, Hi2, LYZ, Lu, St}. In particular, a special K\"ahler manifold carries an affine structure (given by local flat Darboux coordinates), with respect to which the metric is a {\em Hessian metric} \cite{CY} (i.e., its Riemannian metric in local Darboux coordinates is given by the Hessian of a real convex function), see \cite[Proposition 1.24]{Fr}. Recall \cite{CY, Hi1} that a Hessian metric is called a {\em Monge-Amp\`ere metric} if its determinant is a constant, and that an affine structure is called integral if its transition functions have integral linear part \cite{Gro}.

With these definitions in place, we can state our next result, which complements \cite[Theorem 1.3]{GTZ}:
\begin{theorem}\label{theorem2}   If $M$ is an irreducible holomorphic symplectic manifold, and $f:M \to N=\mathbb{CP}^n$ is a holomorphic Lagrangian fibration, then the limiting metric $\omega$ is a special K\"ahler metric on $N_{0}$. Its associated affine structure is integral, and its Riemannian metric is a Monge-Amp\`ere metric on $N_0$.
\end{theorem}

These last two facts are intimately related to the Strominger-Yau-Zaslow picture of mirror symmetry \cite{SYZ}, and to a conjecture of Gross-Wilson \cite[Conjecture 6.2]{GW}, \cite[Conjecture 7.32]{DB}, Kontsevich-Soibelman \cite[Conjecture 1]{KS} and Todorov \cite[p. 66]{Man} which predicts that collapsed Gromov-Hausdorff limits of unit-diameter Ricci-flat K\"ahler metrics on Calabi-Yau manifolds which approach a large complex stucture limit should be half-dimensional Riemannian manifolds on a dense open set. Furthermore, this open set should carry an integral affine structure which makes the metric Monge-Amp\`ere. And finally, the complement of this open set should have real Hausdorff codimension at least $2$.

In our earlier work \cite[Theorem 1.3]{GTZ} we proved the first part of this conjecture for families of hyperk\"ahler manifolds satisfying some hypotheses. Theorem \ref{theorem2} applied to that same setup proves the second part of this conjecture. The third part follows from Theorem \ref{theorem1} in the case when $n=1$, i.e., $M$ is a K3 surface. This was proved by Gross-Wilson \cite{GW} for elliptically fibered  K3 surfaces over $\mathbb{CP}^1$ with  only  singular  fibers of type $\rm I_{1}$, while our results here apply to all elliptically fibered K3 surfaces over $\mathbb{CP}^1$.\\

To prove Theorem \ref{theorem1} we first use Hodge theory to derive precise asymptotics for the fiberwise integrals of the holomorphic volume form on $M$ over the fibers $M_y$ as $y$ approaches a critical value for $f$. This is the content of section \ref{sectvol}. We then use these asymptotics together with estimates from our previous work \cite{GTZ} to complete the proof of Theorem \ref{theorem1} in section \ref{sectpf1}. The proof of Theorem \ref{theorem2} occupies section \ref{sectpf2}.

 \vskip 7mm

\noindent {\bf Acknowledgements:}  The second named author would like to thank Aaron Naber for useful discussions. The last named  author would like to thank Professor Zhiqin Lu, Hao Fang and Jian Song for some discussions. This work was carried out while the last named author was visiting the Mathematics Department of the University of California at San Diego.

\section{Volume asymptotics}\label{sectvol}

In this section we derive asymptotics for the pushforward of the holomorphic volume form on a Calabi-Yau manifold which is the total space of a holomorphic fibration over a curve.

We fix $M$ an $m$-dimensional non-singular
projective Calabi-Yau manifold with holomorphic $m$-form $\Omega$,
$N$ a non-singular projective algebraic curve,
and assume that we have a surjective map $f:M\rightarrow N$.
Let $D\subseteq N$ be the discriminant locus of the map $f$.

By Hironaka's resolution of singularities, there is a birational
morphism $\pi:\tilde M\rightarrow M$ such that $\tilde M$ is
non-singular and $\tilde f:\tilde M\rightarrow N$ is a normal crossings
morphism, i.e., locally
one can find coordinates $(z_1,\ldots,z_m)$ on $\tilde M$
such that $\tilde f$
is given by $(z_1,\ldots,z_m)\mapsto \prod_{i=1}^{m}z_i^{d_i}$.
Note that the fibres may be non-reduced.

Let $U$ be an open neighbourhood in $N$ of a point $y_0\in D$
with $U$ biholomorphic to the unit disc $\Delta$.
Let $y$ be a holomorphic coordinate on $U$ giving this biholomorphism
taking the value $0$ at $y_0$. We write
$\Omega$ also for the pull-back of $\Omega$ on $M$ to $\tilde M$.
Define a real function $\varphi_U$ on $U$ by the identity
\[
 (-1)^{\frac{m^{2}}{2}} \tilde{f}_*\Omega\wedge\bar\Omega= \varphi_U \sqrt{-1}dy\wedge d\bar y.
\]
The main result of this section is the following:

\begin{proposition}\label{volest}
After possibly shrinking $U$, we
obtain the estimate
\begin{equation}
\label{estimate1}
|\varphi_U(y)|\le C|y|^{\alpha} (1-\log |y|)^d
\end{equation}
for some non-negative integer $d$ and rational number $\alpha>-2$ and for all $0\neq y\in U$.
\end{proposition}

\begin{proof}
Let $n_0\in U\setminus\{y_0\}$ be a basepoint and let
\[
T:H^{m-1}(\tilde f^{-1}(n_0),\CC)\rightarrow H^{m-1}(\tilde
f^{-1}(n_0),\CC)
\]
be the monodromy operator for a loop based at $n_0$ around $y_0$.
By the Monodromy Theorem (see e.g., the appendix of \cite{La})
$T$ is quasi-unipotent with $(T^{\beta}-I)^{d}=0$ for some
positive integers $d$ and $\beta$, and $\beta$ is the least common multiple
of the multiplicities of the irreducible components of the
fibre over $y_0$. Let $\bar U=\Delta$ with coordinate $w$,
and let $\mu:\bar U\rightarrow U$
be given by $\mu(w)=w^\beta$.

Pull-back and normalize the family
$\tilde M\rightarrow N$
via the composition $\bar U
{\smash{
 \mathop{\longrightarrow}\limits^{\mu}}}
 U\hookrightarrow \tilde N$,
to obtain a family $\bar f:\bar M\rightarrow \bar U$.
This has discriminant locus $\bar D=\mu^{-1}(y_0)=\{0\}$
and $\bar f$ now has the property that the monodromy around a
loop in $\bar U^o:=\bar U\setminus \mu^{-1}(y_0)$
is unipotent.

Now the trivial vector bundle
$\mathcal{ H}^{m-1}=(R^{m-1}\bar f_*\CC)\otimes_{\CC}\mathcal{O}_{\bar U^o}$ on
$\bar U^o$ comes with
the Gauss-Manin connection, whose flat sections are sections of
$R^{m-1}\bar f_*\CC$. It is standard that this vector bundle has
a canonical extension to $\bar U$, (see e.g., \cite{Gr}, Chapter IV)
constructed as follows.
Choosing a basepoint $t_0\in \bar U^o$, let $e_1,\ldots,e_s$ be
a basis for $H^{m-1}(\bar f^{-1}(t_0))$. These extend to multi-valued flat
sections of $\mathcal{H}^{m-1}$, which we write as $e_i(w)$. However,
\[
\sigma_i(w):=\exp\left(-N\frac{\log w}{2\pi\sqrt{-1}}
\right)e_i(w)
\]
with $N=\log T$ is in fact a single-valued holomorphic section of
$\mathcal{H}^{m-1}$. We then extend $\mathcal{H}^{m-1}$ across $\bar U$ by decreeing
these sections to form a holomorphic frame for the vector bundle. Call
this extension $\mathcal{H}^{m-1}_{\bar U}$

It is then standard (see again \cite{Gr}, Chapter IV)
that the Hodge bundle
$F^{m-1}_{\bar U^o}:=(f_*\Omega^{m-1}_{\bar M/\bar U})|_{\bar U^o}\subseteq
\mathcal{H}^{m-1}$
has a natural extension $F^{m-1}_{\bar U}\subseteq \mathcal{H}^{m-1}_{\bar U}$ to
$\bar U$.

Next note that the form
\[
\Omega^{rel}:=\iota(\partial/\partial y)\Omega
\]
is a well-defined section of $\tilde f_*\Omega^{m-1}_{\tilde f^{-1}(U)/U}$
and thus pulls back to
a well-defined section $\Omega^{rel}_{\bar U^o}$ of
$F^{m-1}_{\bar U^o}$. Furthermore,
the function $\varphi_U$ given in the statement
of the theorem satisfies at a point $y\in U$
\[
\varphi_U(y)=(-1)^{\frac{(m-1)^{2}}{2}}\int_{\tilde f^{-1}(y)} \Omega^{rel}\wedge \bar\Omega^{rel}.
\]

We will show that the section $\Omega^{rel}_{\bar U^o}$ of $F^{m-1}_{\bar U^o}$
extends to a meromorphic section of $F^{m-1}_{\bar U}$ and investigate
the order of the pole of this section at $\mu^{-1}(y_0)$. 

Let $\tilde Y$ denote the fibre of $\tilde f:\tilde M\rightarrow N$ over
$y_0$.
We first determine the order of pole of $\Omega^{rel}$ at $0$ as a section
of $\Omega^{m-1}_{\tilde M/N}(\log\tilde Y)$.
Locally on $\tilde M$, near a general point of an irreducible
component of $\tilde Y$, the map $\tilde f$ is given by
$y=z_1^\ell$, with $\ell\ge 1$ and $z_1,\ldots,z_{m}$ coordinates
on $\tilde M$. We can write $\Omega$ as a form on $\tilde M$ locally as
\[
\Omega:=\psi dz_1\wedge\cdots\wedge dz_m
\]
for some holomorphic function $\psi$. In our local coordinate description,
the vector field on $\tilde M$ given by $\ell^{-1}z_1^{-\ell+1}
\partial_{z_1}$
is a lift of $\partial_{y}$. Thus $\Omega^{rel}$ as a section
of $\Omega^{m-1}_{\tilde M /N}(\log\tilde Y)$
is locally given by
\[
\pm\frac{\psi}{\ell z_1^{\ell-1}}dz_2\wedge \cdots
\wedge dz_{m}.
\]
This shows that we can consider
$\Omega^{rel}$ as a meromorphic section of
$\Omega^{m-1}_{\tilde M/
N}(\log\tilde Y)$, hence of $\tilde f_*\Omega^{m-1}_{\tilde M/N}(\log\tilde Y)$.

We now need to pull-back $\Omega^{rel}$ to $\Omega^{rel}_{\bar U^o}$ and
study this section as a section of $F^{m-1}_{\bar U}$. To this end, we note
that the 
stable reduction theorem \cite{KKMS} gives a resolution of singularities
$\bar M'\rightarrow \bar M$ such that the composed map
$\bar M'\rightarrow \bar U$ is normal crossings. So we have a diagram
\[
\xymatrix@C=30pt
{ \bar M'\ar[r]^{\pi'}\ar[rd]_{\bar f'} & \bar M\ar[d]_{\bar f}
\ar[r]^{\pi}&\tilde M\ar[d]^{\tilde f}\\
&\bar U\ar[r]_{\mu}&N}
\]
Furthermore, the map $\pi'$ is a toric resolution of singularities
by the construction of \cite{KKMS}, Chapter II. In particular, locally
$\pi'$ and $\pi$ can be described as dominant morphisms of toric varieties
of the same dimension. On such toric charts, by \cite{Oda}, Prop.\ 3.1,
the sheaves of
logarithmic differentials
$\Omega^{m-1}_{\tilde M/N}(\log \tilde Y),
\Omega^{m-1}_{\bar M/\bar U}(\log \bar Y)$, and
$\Omega^{m-1}_{\bar M'/\bar U}(\log \bar Y')$ are trivial vector bundles
generated by exterior products of logarithmic differentials of toric
monomials, and thus
$\pi^*\Omega^{m-1}_{\tilde M/N}(\log \tilde Y)\cong
\Omega^{m-1}_{\bar M/\bar U}(\log \bar Y)$
and $(\pi')^*\Omega^{m-1}_{\bar M/\bar U}(\log \bar Y)
\cong \Omega^{m-1}_{\bar M'/\bar U}(\log \bar Y')$.
Furthermore,
\begin{align*}
\bar f'_*\Omega^{m-1}_{\bar M'/\bar U}(\log \bar Y')
\cong {} &
\bar f_*\pi'_*\Omega^{m-1}_{\bar M'/\bar U}(\log \bar Y')\\
\cong {} & \bar f_*\pi'_*(\pi')^*\Omega^{m-1}_{\bar M/\bar U}(\log\bar Y)\\
\cong {} & \bar f_*((\pi'_*\mathcal{O}_{\bar M'})\otimes
\Omega^{m-1}_{\bar M/\bar U}(\log\bar Y))\\
\cong {} & \bar f_*\Omega^{m-1}_{\bar M/\bar U}(\log \bar Y).
\end{align*}
It also follows from \cite{Steen} (see also \cite{Gr}, Chapter VII)
that
\[
F^{m-1}_{\bar U}
\cong \bar f'_*\Omega^{m-1}_{\bar M'/\bar U}(\log \bar Y').
\]
Thus, in order to understand the behaviour of $\Omega^{rel}_{\bar U^o}$
as a section of $F^{m-1}_{\bar U}$, 
it is sufficient to pull back $\Omega^{rel}$ to
$\Omega^{m-1}_{\bar M/\bar U}(\log \bar Y)$ and understand the behaviour
of this form as a section of $\bar f_*\Omega^{m-1}_{\bar M/\bar U}
(\log \bar Y)$.

Again, we do this locally near the inverse image of a general point of
an irreducible component of $\tilde Y$. Using the same notation as before,
we know that $\bar M$ is locally given by the normalization of
the equation $w^{\beta}=z_1^\ell$. Note that $\ell|\beta$,
so a local description of the normalization is given by an equation
$w^{\beta/\ell}=\xi z_1$ for $\xi$ an $\ell$-th root of unity. Thus
$\Omega^{rel}$ pulls back to
\[
C\cdot \psi w^{\frac{-\beta(\ell-1)}{\ell}}dz_2\wedge\cdots\wedge dz_{\beta}.
\]
Thus letting $\ell$ be the largest multiplicity of any irreducible
component of $\tilde Y$, we find $w^{\beta(\ell-1)/\ell}\Omega^{rel}_{\bar U^o}$
extends to a holomorphic section of $\Omega^{m-1}_{\bar M/\bar U}(\log \bar Y)$,
hence yields a holomorphic section of
$F^{m-1}_{\bar U}=
\bar f'_*\Omega^{m-1}_{\bar M'/\bar U}(\log \bar Y')$.

Now set
\[
\Omega^{norm}:= w^{\beta(\ell-1)/\ell} \Omega^{rel}_{\bar U^o}.
\]
This now extends to a holomorphic section of $F^{m-1}_{\bar U}$.
Thus we can write $\Omega^{norm}$, as a section of $\mathcal{H}^{m-1}_{\bar U}$,
as
\[
\Omega^{norm}=\sum_{i=1}^s h_i(w)\sigma_i(w),
\]
for $h_i$ holomorphic functions on $\bar U$. We then compute, with
$\langle\cdot,\cdot\rangle$ denoting the cup product followed by evaluation
on the fundamental class
\[
H^{m-1}(f^{-1}(t_0),\CC)\times H^{m-1}(f^{-1}(t_0),\CC)\rightarrow
\CC,
\]
that
\begin{align*}
&\int_{\bar f^{-1}(w)}\Omega^{norm}\wedge \bar\Omega^{norm}\\
= {} & \left\langle \sum_{i=1}^s h_i(w)\sigma_i(w), \sum_{j=1}^s \bar h_j(w)
\bar \sigma_j(w)\right\rangle\\
= {} &
\left\langle
\sum_{i=1}^se^{- N\log w/2\pi\sqrt{-1}} h_ie_i,
\sum_{j=1}^se^{ N\log \bar w/2\pi\sqrt{-1}}
\bar h_je_j\right\rangle.
\end{align*}
Note the exponentials can be expanded in a finite power series
because $N$ is nilpotent, and hence a term in the above expression
is
\[
C\cdot h_i\bar h_j (\log w)^{d}(\log \bar w)^{d'}
\left\langle N^{d}e_i, N^{d'}e_j\right\rangle.
\]
Here the constant $C$ only depends on the powers $d$, $d'$ occuring.
We can assume that we have chosen the imaginary part of $\log w$
to lie between $0$ and $2\pi$ (this is equivalent to choosing the branch
of $e_i(y)$). Keeping in mind that the $h_j$ are holomorphic on $\bar U$,
after shrinking $\bar U$ we can assume that $|h_j|$ are bounded by
some constant, and so we see that the above term is bounded by a sum of a finite
number of expressions of the form
\[
C'(-\log |w|)^{d''}.
\]
Thus the entire integral is bounded by an expression of the form
\[
C(1- \log |w|)^d
\]
for suitable choice of constant $C$ and exponent $d$.

Returning to $\Omega^{rel}_{\bar U^o}$, we see that
\[
(-1)^{\frac{(m-1)^{2}}{2}}\int_{\bar f^{-1}(w)}\Omega^{rel}_{\bar U^o}\wedge \bar\Omega^{rel}_{\bar U^o}
\le C|w|^{-2\beta(\ell-1)/\ell}(1-\log |w|)^d.
\]
Using $y=w^{\beta}$ then gives the result.
\end{proof}

In fact, the volume asymptotics we just proved can be easily generalized to the case when the base $N$ has arbitrary dimension $m$, but these estimates are not enough for the arguments in section \ref{sectpf1} to go through.

\section{Proof of Theorem \ref{theorem1} }\label{sectpf1}
In this section we use the results for section \ref{sectvol} and give a proof of Theorem \ref{theorem1}.

Therefore assume that we are in the setting of Theorem \ref{theorem1}.
As explained for example in \cite[Section 4]{To1}, (especially equation (4.3))
the limiting metric $\omega$ satisfies the equation
\begin{equation}\label{eq3} \omega =c_{1}\frac{\alpha_{0}\cdot \alpha^{m-1}}{(\alpha_{0}+\alpha)^{m}}f_{*}((-1)^{\frac{m^{2}}{2}}\Omega \wedge \overline{\Omega}) \end{equation}  on $N_{0}$.
Now $N$ is a compact Riemann surface, and therefore the discriminant locus  $D=N\backslash N_{0}=\cup_{k=1}^{|D|} \{ p_{k}\}$ is a finite set (here $|D|$ denotes the cardinality of $D$). For any $p_{k}$,  there is a neighborhood $U_{k}$ such that $U_{k}$ admits a coordinate $y$,  $p_{k}$ is given by $y=0$,  and  $p_{\ell}$ does not belong to  $ U_{k}$ if $\ell\neq k$.  Set $U_k^*=U_k-\{p_k\}$.  
Let  $\Omega_{y,k}$ be a holomorphic relative volume form on $f^{-1}(U_k^*)$
(i.e., a nowhere vanishing holomorphic section of the relative canonical bundle $K_{f^{-1}(U_{k}^*)/U_{k}^*}$) such that  we have
\begin{equation}\label{eq4}
\Omega=f^{*}dy  \wedge \Omega_{y, k}.
\end{equation}   on $f^{-1}(U_k^{*})$.
Then on $U^{*}_k$ we have
\begin{eqnarray}\label{eq2}
f_{*}((-1)^{\frac{m^{2}}{2}}\Omega \wedge \overline{\Omega})& = & (-1)^{\frac{m^{2}}{2}+m-1}\left ( \int_{M_{y}}  \Omega_{y, k} \wedge  \overline{\Omega}_{y, k}\right) dy\wedge d\overline{ y}\\ & \leqslant &  C \left((-1)^{\frac{(m-1)^{2}}{2}}\int_{M_{y}}  \Omega_{y, k} \wedge  \overline{\Omega}_{y, k}\right) \omega_{0},
\end{eqnarray}
for $C>0$ a fixed constant.
Denote
\begin{equation}
\label{eq5}
\varphi_{U_{k}}(y)= (-1)^{\frac{(m-1)^{2}}{2}}\int_{M_{y}}  \Omega_{y, k} \wedge  \overline{\Omega}_{y, k}.
\end{equation}
Thanks to \eqref{eq3}, \eqref{eq4} and \eqref{eq2}, there is a constant $C>0$ such that
\begin{equation}
\label{eq6}
\omega \leqslant C \varphi_{U_{k}}(y) \omega_{0},
\end{equation} holds on $U^{*}_{k}$.
For $0< \rho\leq e^{-1}$, denote $\Delta_{k}^{*}(\rho)=\{y\in U_{k}\,|\, 0< |y|<\rho\}$.

\begin{lemma}\label{compl}  For each given $1\leq k\leq |D|$, there are constants $C>0$, $d\in\mathbb{N}$, $\alpha\in\mathbb{Q}$ with $\alpha>-2$, such that for any $\rho>0$ sufficiently small and for any two points $q_{1}$ and $q_{2}\in \Delta_{k}^{*}(\rho)$,  there is a curve $\gamma \subset \Delta_{k}^{*}(\rho)$ connecting $q_{1}$ and $q_{2}$ such that $$ {\rm length}_{\omega} (\gamma) \leqslant C \rho^{1+\frac{\alpha}{2}} (-\log \rho)^{d}.$$  Furthermore, the metric completion of $(N_0,\omega)$ is a compact metric space homeomorphic to $N$.
\end{lemma}

 \begin{proof} Proposition \ref{volest} applied to $U_k$, shows that there is a small $\rho>0$ (which we can assume is less than $e^{-1}$) and there
are constants $C>0$, $d\in\mathbb{N}$, $\alpha\in\mathbb{Q}$ with $\alpha>-2$, such that on $\Delta_{k}^{*}(\rho)$ we have
$$\varphi_{U_{k}}(y)\leq C |y|^{\alpha} (-\log |y|)^{d}.$$
Together with \eqref{eq6}, this gives
$$ \omega \leqslant C  |y|^{\alpha} (-\log |y|)^{d}\omega_{0},$$ on $ \Delta_{k}^{*}(\rho)$.
Here and in the following, we use $C$ to denote a fixed positive constant, which may change from line to line.
For any $0\leq\theta\leq 2\pi$, we set $\gamma_{\theta}(s)= s e^{\sqrt{-1}\theta}$, which for $0<s<\rho$ gives a path in $\Delta_{k}^{*}(\rho)$.
We also set $S^{1}(\rho)=\{y \in \Delta_{k}^{*}(\rho)\ |\quad |y|=\rho\}$. Then we use the fact that $\rho< e^{-1} $ so that $-\log\rho\geq -\log e^{-1}  \geqslant 1$,
and use repeated integration by parts to compute
\[\begin{split}
{\rm length}_{\omega}  (\gamma_{\theta})  &\leqslant  C \int_{0}^{\rho}\sqrt{s^{\alpha} (-\log s)^{d}} ds \leq C'\int_0^\rho s^{\frac{\alpha}{2}} (-\log s)^d ds\\
&\leq C''\rho^{1+\frac{\alpha}{2} }(-\log \rho)^{d}.
\end{split}\]
We also have that
$$ {\rm length}_{\omega}  (S^{1}(\rho))  \leqslant  C \rho^{1+\frac{\alpha}{2}}(-\log \rho)^{\frac{d}{2}}\leqslant   C \rho^{1+\frac{\alpha}{2}}(-\log \rho)^{d}. $$
Hence for any $q_{1}, q_{2} \in \Delta_{k}^{*}(\rho)$,  there is a curve $\gamma \subset \Delta_{k}^{*}(\rho)$ connecting $q_{1}$ and $ q_{2}$ such that
\begin{equation}\label{to0}
{\rm length}_{\omega} (\gamma) \leqslant C  \rho^{1+\frac{\alpha}{2}}(-\log \rho)^{d}  \to 0,
\end{equation}
when $\rho \to 0$, because $\alpha>-2$.
By repeating this argument near each point $p_k$, it follows that $$\sup_{y_{1},y_{2}\in N_{0}} d_{\omega}(y_{1},y_{2}) \leqslant C,$$ where
$d_{\omega}$ is the metric space structure   on $N_{0}$ induced by $\omega$.

For any $q\in N_{0}$,  $p_{k}\in D$,  and a sequence $q_{s} \to p_{k}$, the distance  $d_{\omega}(q, q_{s}) $ converges by passing to a subsequence. Define $$d_{\omega}(q, p_{k})=\lim_{s\to \infty} d_{\omega}(q, q_{s}).$$ If  $q_{s}' \to p_{k}$ is another sequence, then by passing to subsequences we can assume that $q_{s}$ and $q_{s}'\in \Delta_{k}^{*}(\frac{1}{s})$ for all $s$ large. Thanks to \eqref{to0}, we see that $d_{\omega} (q_{s}, q_{s}') \to 0$.  Thus $d_{\omega}(q, p_{k})$ does not depend on the choice of the sequence $\{q_{k}\}$, and is well-defined. Similarly, we can define $d_{\omega}(p_{\ell}, p_{k})$ for any  $p_{\ell}, p_{k} \in D$, and we clearly have $d_\omega(p_\ell, p_k)>0$ for $k\neq \ell$.   We obtain that $d_{\omega}$ is a compact metric space structure on (a space homeomorphic to) $N$ which by definition is the metric completion of $(N_0,\omega)$.
  \end{proof}

From now on we will denote by $(N,d_\omega)$ the metric completion of $(N_0,\omega)$. Now pick any sequence $t_k\to 0$ such that $(M,\ti{\omega}_{t_k})$ converges in the Gromov-Hausdorff topology to a compact length metric space $(X,d_X)$. As we recalled in the Introduction, in \cite[Theorem 1.2]{GTZ} we constructed a local isometric embedding
of $(N_0,\omega)$ into $(X,d_X)$ with open dense image $X_0\subset X$ via a homeomorphism $\phi:N_0\to X_0$.

\begin{lemma}\label{finite}
The set $S_X=X\backslash X_0$ is finite, with cardinality less than or equal to the cardinality of $D$.
\end{lemma}
\begin{proof}
The density of $X_0$ implies that for every fixed $\rho>0$ the set
$$\bigcup_{k=1}^{|D|}\{\overline{\phi(\Delta_{k}^{*}(\rho))}\}$$ covers $S_{X}$. Then the fact that $(X,d_X)$ is a length space implies that
$${\rm diam}_{d_{X}} (\overline{\phi(\Delta_{k}^{*}(\rho))}) ={\rm diam}_{d_{X}} (\phi(\Delta_{k}^{*}(\rho)))=\sup_{p,q\in \phi(\Delta_{k}^{*}(\rho))}\inf_{\sigma}{\rm length}_{d_X}(\sigma),$$
where the infimum is over all curves $\sigma$ in $X$ joining $p$ and $q$. But $p$ and $q$ can be joined by curves of the form $\phi(\gamma)$, with $\gamma\subset \Delta_{k}^{*}(\rho)$,
and since $\phi$ is a local isometry we have that
\begin{equation}\label{lengthp}
{\rm length}_{d_X}(\phi(\gamma))={\rm length}_{\omega}(\gamma),
\end{equation}
for any such curve $\gamma$.
We can then use \eqref{to0} and conclude that
$${\rm diam}_{d_{X}} (\overline{\phi(\Delta_{k}^{*}(\rho))}) \leqslant C  \rho^{1+\frac{\alpha}{2}}(-\log \rho)^{d},$$
for all $\rho>0$ small and for $C>0$ independent of $\rho$. Since this approaches $0$ as $\rho\to 0$, we conclude that the
$0$-dimensional Hausdorff measure of
$$S_X\cap \bigcap_{\rho>0} \overline{\phi(\Delta_{k}^{*}(\rho))}$$ equals $1$, and so this set is a single point $x_k$. Therefore $S_X=\cup_{k=1}^{|D|} \{x_k\}$ is a finite set. Note that a priori
the points $x_k$ need not all be distinct.
\end{proof}

The map $\phi$ extends to a surjective  $1$-Lipschitz  map $\tilde{\phi}: N \to  X$ by letting $ \tilde{\phi}(p_{k})=x_{k}$.  Furthermore, if $\mathcal{H}^{\beta}_{d_{X}}$
denotes   the $\beta$-dimensional Hausdorff measure on $X$,   then    \begin{equation}
\label{measure}  \mathcal{H}^{\beta}_{d_{X}} (S_{X})=0,  \  \  \  {\rm and} \  \  \  \mathcal{H}^{2}_{d_{X}} (X)= {\rm Vol}_{\omega} (N_{0}),
 \end{equation} for $0< \beta \leqslant 2$.

Note also that for any two points $x,y\in N_0$ we have
\begin{equation}\label{length}
d_X(\phi(x),\phi(y))\leq d_\omega(x,y).
\end{equation}
Indeed, for any $\ve>0$ there is a path $\gamma_\ve$ in $N_0$ joining $x$ and $y$ with
$\mathrm{length}_{\omega}(\gamma_\ve)\leq d_\omega(x,y)+\ve$. From \eqref{lengthp} we see that
$$\mathrm{length}_{\omega}(\gamma_\ve)=\mathrm{length}_{d_X}(\phi(\gamma_\ve))\geq d_X(\phi(x),\phi(y)),$$ and letting $\ve\to 0$ proves \eqref{length}.

\begin{lemma}\label{isom} The map  $\tilde{\phi}: (N, d_\omega) \to  (X, d_{X})$ is an isometry.
In particular we have $x_{k} \neq x_{\ell}$ for $k\neq \ell$.
\end{lemma}

\begin{proof}
If $\nu$ denotes the reduced measure in Section 5 of  \cite{GTZ}, then $\nu(S_{X})=0$ \cite[Remark 5.3]{GTZ} and there exists a constant $\upsilon>0$ such that for any $K\subset N_{0}$, $$\nu(K)=\upsilon \int_{f^{-1}(K)} (-1)^{\frac{m^{2}}{2}}\Omega \wedge \overline{\Omega}=\upsilon \int_{K} f_{*}(-1)^{\frac{m^{2}}{2}}\Omega \wedge \overline{\Omega}=\frac{\upsilon (\alpha_{0}+\alpha)^{m}}{c_{1}\alpha_{0}\cdot \alpha^{m-1}}
\int_{K}\omega$$ by (\ref{eq3}).  Thus  $$\nu(K)=\lambda {\rm Vol}_{\omega}(K)=\lambda\mathcal{H}^{2}_{d_{X}}(K),$$ for a constant $\lambda>0$.
Furthermore, $\nu$ is a Radon measure by \cite[Theorem 1.10]{CC1}.
It follows that for any Borel set $A\subset X$ we have
$$\nu(A)=\nu(A\cap S_X)+\nu(A\backslash S_X)=\nu(A\backslash S_X),$$ and if we pick an exhaustion of $A\backslash S_X$
by compact sets $A_k\subset N_0$ then we have
$$\nu(A)=\lim_{k\to\infty}\nu(A_k),$$
because every Radon measure is inner regular.
But since $A_k$ is relatively compact in $N_0$, we have $\nu(A_k)=\lambda \mathcal{H}^{2}_{d_{X}}(A_k)$ for all $k$,
and so
$$\nu(A)=\lambda \lim_{k\to\infty}\mathcal{H}^{2}_{d_{X}}(A_k)=\lambda \mathcal{H}^{2}_{d_{X}}(A\backslash S_X),$$
because the Hausdorff measure $\mathcal{H}^{2}_{d_{X}}$ is inner regular. Using \eqref{measure}, we conclude that
$$\nu(A)=\lambda\mathcal{H}^{2m}_{d_{X}}(A),$$
i.e., that $\nu= \lambda \mathcal{H}^{2}_{d_{X}}$ as measures.
Recall from \cite[Section 2]{CC2} the following construction: given a Borel measure $\mu$ on a metric space $(Z,d)$
and a real number $\beta$ the Hausdorff measure in codimension $\beta$ is defined by
$$\mu_{-\beta}(U)=\lim_{\delta\to 0}(\mu_{-\beta})_{\delta}(U),$$
for all subsets $U$ of $Z$ where
$$(\mu_{-\beta})_{\delta}(U)=\inf \sum_i r_i^{-\beta}\mu(B_i),$$
and the infimum is over all coverings of $U$ by balls $B_i$ of radii $r_i<\delta$.
Then $\mu_{-\beta}$ is a metric outer measure, whose associated Radon measure is also denoted by $\mu_{-\beta}$.
For example if $\mu=\mathcal{H}^{2}_{d_{Z}}$ then $\mu_{-\beta}$ is uniformly equivalent to $\mathcal{H}^{2-\beta}_{d_{Z}}$.

Then from the equality of measures $\nu= \lambda \mathcal{H}^{2}_{d_{X}}$ we deduce that $$\nu_{-1}(S_{X})=\lambda  \mathcal{H}^{1}_{d_{X}}(S_{X})=0,$$
because $S_X$ is a finite set.
We can then apply \cite[Theorem 3.7]{CC2}, and see that given any $x_1\in X_0$ for $\mathcal{H}^2_{d_X}$-almost all $y\in X_0$ there exists a minimal geodesic from $x_1$ to $y$ which lies entirely in $X_0$. In particular, given any two points $x_1,x_2\in X_0$ and $\delta>0$, there is a point $y\in X_0$ with $d_X(x_2,y)<\delta$ which can be joined to $x_1$ by a minimal geodesic $\sigma_1$ contained in $X_0$. Furthermore we can take $y$ close enough to $x_2$ so that it can also be joined to $x_2$ by a curve $\sigma_2$ contained in $X_0$ with $d_X$-length at most $\delta$. Concatenating $\sigma_1$ and $\sigma_2$ we obtain a curve $\sigma$ in $X_0$ joining $x_1$ to $x_2$ with
$${\rm length}_{d_X}(\sigma)\leqslant d_{X} (x_{1}, y)+ \delta\leq d_X(x_1,x_2)+2\delta.$$
Since $\phi:N_0\to X_0$ is a homeomorphism, we conclude that given any two points $q_1,q_2\in N_0$ and $\delta>0$, there is a curve $\gamma$ in $X_0$ joining $q_1$ and $q_2$ with
$${\rm length}_\omega(\gamma)\leq d_{X} (\phi(q_{1}), \phi(q_{2}))+ 2\delta.$$
Therefore, thanks to \eqref{length}, we conclude that
$$d_{X} (\phi(q_{1}), \phi(q_{2})) \leqslant d_{\omega} (q_{1}, q_{2})\leqslant{\rm length}_{\omega} (\gamma) \leqslant d_{X} (\phi(q_{1}), \phi(q_{2}))+ 2\delta.$$
Letting $\delta \to 0$, we conclude that
$$d_{\omega} (q_{1}, q_{2}) =  d_{X} (\phi(q_{1}), \phi(q_{2})).$$ Thus the extension $\tilde{\phi}$ is an isometry between $(N, \omega)$ and $(X, d_{X})$.
 \end{proof}
From Lemma \ref{isom} and Gromov's pre-compactness theorem, we immediately conclude that
$$(M, \tilde{\omega}_{t}) \stackrel{d_{GH}}\longrightarrow  (N, d_\omega),$$
as $t\to 0$ without passing to subsequences.
Putting together Lemmas \ref{compl}, \ref{finite} and \ref{isom} completes the proof of Theorem \ref{theorem1}.

 \section{ Proof of Theorem \ref{theorem2} }\label{sectpf2}
In this section we give the proof of Theorem \ref{theorem2}.

Let $(M,  \Theta)$ be an irreducible holomorphic symplectic manifold of dimension $2n$ as in the hypotheses.   By normalizing $\Theta$, we assume that $\lim_{t\to 0}c_{t}=1$.  Then $( \tilde{\omega}_{t}, c_{t}^{\frac{1}{2n}}\sqrt{t}\Theta)$  is a hyperk\"ahler structure for any $t$, and
$$\ti{\omega}_t^{2n}=c_t t^{n}\Theta^n\wedge\overline{\Theta}^n.$$
Note that $f:  f^{-1}(N_{0}) \to N_{0}$ with holomorphic symplectic form $\Theta$ and polarization $\alpha$ is an algebraic integrable system as in \cite{Fr}.  By Section 3 in \cite{Fr}, there is a $\mathbb{Z}^{2n}$-lattice subbundle  $\Lambda \subset  T^{*}N_{0}$  such that  $\Lambda$ is a holomorphic  Lagrangian submanifold with respect to the canonical holomorphic symplectic form $\hat{\Theta}$ on $T^*N_0$.
The holomorphic symplectic form $\hat{\Theta}$ induces a holomorphic symplectic form  on $T^{*}N_{0}/ \Lambda$, still denoted by $\hat{\Theta}$.    Furthermore, for any local Lagrangian   section $s: B \to f^{-1}(B)$ where $B\subset N_{0}$,   there is 
a biholomorphism $\Phi_{s}:  T^{*}B/\Lambda  \to f^{-1}(B)$  such that  $\Phi_{s}^{*}\Theta= \hat{\Theta}$, and $\Phi^{-1}_{s}  (s(B)) $ is the zero section.    Since any two biholomorphisms induced by local sections differ by a translation on each fiber,  the polarization $\alpha$ induces a polarization $\hat{\alpha}$ on  $T^{*}N_{0}/ \Lambda$.   By choosing a local section $s: B \to f^{-1}(B)$, we will not distinguish between $T^{*}N_{0}/\Lambda |_{B}$ and $  f^{-1}(B)$,  $\Theta$ and $ \hat{\Theta}$, $\alpha$ and  $\hat{\alpha}$.

Let $B\subset N_{0}$ be a subset biholomorphic to a polydisc $\Delta^{n}$, and $U=f^{-1}(B)$.
The polarization $\alpha$ induces a symplectic basis $\delta_{1}, \cdots, \delta_{n}, \xi_{1}, \cdots, \xi_{n} $ of $T^*B$, which generates $\Lambda$. If $x_{i}, x_{n+i}$ denote the coordinates  corresponding to the basis  $\delta_{i}, \xi_{i} $,   i.e.,  $ T^*_{p}B= \{\sum x_{i} \delta_{i} +x_{n+i} \xi_{i} |   x_{i}, x_{n+i} \in \mathbb{R}\}$ for any $p\in B$,   then there exist positive integers $d_{i}$ such that
$$\sum_{1\leqslant i\leqslant n} d_{i}(dx_{i}\wedge dx_{n+i})|_{M_{y}}$$
is a K\"ahler metric on $M_y$ which belongs to the class $\alpha |_{M_{y}},$
for any  $y\in B$.  If we denote  $e_{i}=d_{i}^{-1}\delta_{i}$, then  $e_{i}$ is a Lagrangian section of $T^{*}B$, which implies that there are coordinates $y_{1}, \ldots, y_{n}$ on $B$ such that   $e_{i}={\rm Re}\, dy_{i}$.  Let  $z_{1}, \ldots, z_{n}$ be the coordinates corresponding to  $e_{i}$, i.e., the identification $T^{*}B\cong B\times \mathbb{C}^{n}$ is given by $\sum z_{i}dy_{i} \mapsto (y, z_{1}, \ldots, z_{n})$.    Then  $$\Theta=\sum_{1\leqslant i\leqslant n} dz_{i}\wedge dy_{i}, \    \   \  {\rm and} \  \   \  \Lambda={\rm span} \{ d_{1}e_{1}, \ldots, d_{n}e_{n}, Z_{1}, \ldots, Z_{n}\}$$ where  $Z_i:B\rightarrow \mathbb{C}^n$ is a holomorphic map for each $i$ and we
wrrite $Z=(Z_1,\dots,Z_n)$.
Furthermore, the Riemann bilinear relations say that  $Z^{T}=Z$ and ${\rm Im}\, Z>0$.

Let $\omega_{SF}$ be the semi-flat form constructed in  \cite[Section 3]{GTZ}, which satisfies $$\omega_{SF}|_{M_{y}}=\sqrt{-1}\sum ({\rm Im}\,Z)^{-1}_{ij}dz_{i}\wedge d\overline{z}_{j}\in \alpha|_{M_{y}}$$ where $y\in B$.  Following \cite{GTZ}, we define $\lambda_{t}: B\times \mathbb{C}^{n} \to B\times \mathbb{C}^{n}$ by $\lambda_{t}(y,z)=(y, t^{-\frac{1}{2}}z)$ and $p: B\times \mathbb{C}^{n} \to T^{*}B/\Lambda $.  It is proved in \cite[Lemma 4.7]{GTZ} that $$\lambda_{t}^{*}p^{*}T^{*}_{\sigma} \tilde{\omega}_{t} \to p^{*}(\omega_{SF}+f^{*}\omega),$$ in the $C^{\infty}$ topology where $\sigma: B \to U$ is a holomorphic section and $T_\sigma:U\to U$ is the fiberwise translation by $\sigma$.
Note that
\[\begin{split}
\lambda_{t}^{*}p^{*}T^{*}_{\sigma} \sqrt{t}\Theta&=\sqrt{t}\left(\sum d\left(\frac{z_{i}}{\sqrt{t}}\right)\wedge dy_{i}+\sum d\sigma_{i}(y)\wedge dy_{i}\right)\\
&=\Theta+\sqrt{t}\sum d\sigma_{i}(y)\wedge dy_{i}.
\end{split}\]
Thus $$\lambda_{t}^{*}p^{*}T^{*}_{\sigma} \sqrt{t}\Theta \to \Theta,$$ in the $C^{\infty}$ topology and $$  (\omega_{SF}+f^{*}\omega)^{2n}=\Theta^{n}\wedge\overline{\Theta}^{n}.$$
Therefore $(\omega_{SF}+f^{*}\omega, \Theta)$ is a hyperk\"ahler structure on $T^{*}B$.

\begin{lemma} On $B$ we have $$\omega=\sqrt{-1}\sum ({\rm Im}\,Z)_{ij}dy_{i}\wedge d\overline{y}_{j}=\sum d_{i}^{-1}\delta_{i}\wedge \xi_{i},$$ and $\omega$ is a special K\"ahler metric on $N_{0}$.
\end{lemma}

\begin{proof}
Denote $\omega=\sqrt{-1}\sum A_{ij}dy_{i}\wedge d\overline{y}_{j}$,  $C=({\rm Im}\,Z)^{-1}$,  and $g$ the hyperk\"ahler metric of $(\omega_{SF}+f^{*}\omega, \Theta)$.  If $x$ belongs the zero section $\Phi^{-1}_{s} (s(B))$,  then, on $ T_{x}(T^{*}B)$, $$\omega_{SF}=\sqrt{-1}\sum C_{ij}dz_{i}\wedge d\overline{z}_{j}, \  \  \  {\rm Re}\,\Theta=\sum (dz_{i}'\wedge dy_{i}'-dz_{i}''\wedge dy_{i}'')$$  $$g=\sum C_{ij}(dz_{i}'dz_{j}'+dz_{i}''dz_{j}'')+\sum A_{ij}(dy_{i}'dy_{j}'+dy_{i}''dy_{j}''),$$ where $z_{i}=z_{i}'+\sqrt{-1}z_{i}''$ and $y_{i}=y_{i}'+\sqrt{-1}y_{i}''$.  We can
define one of the complex structures $J$ on $T^*B$ compatible with the
metric $g$ by ${\rm Re}\,\Theta  (\cdot , \cdot)=g(\cdot, J \cdot)$.
We calculate that $J$ acts as
$$\frac{\partial}{\partial z_{i}'}\mapsto -\sum A^{-1}_{ij}\frac{\partial}{\partial y_{j}'}, \  \   \  \  \   \   \  \frac{\partial}{\partial z_{i}''}\mapsto \sum A^{-1}_{ij}\frac{\partial}{\partial y_{j}''}, $$ $$\frac{\partial}{\partial y_{i}'}\mapsto \sum C^{-1}_{ij}\frac{\partial}{\partial z_{j}'},  \  \   \  \  \   \   \   \frac{\partial}{\partial y_{i}''}\mapsto -\sum C^{-1}_{ij}\frac{\partial}{\partial z_{j}''}.$$ From $J^{2}=-{\rm id}$, we obtain $A=C^{-1}={\rm Im}\,Z,$ and $$\omega=\sqrt{-1}\sum ({\rm Im}\,Z)_{ij}dy_{i}\wedge d\overline{y}_{j}.$$
Note that $e_{i}=d_{i}^{-1}\delta_{i}$ and $\xi_{i}= \sum({\rm Re}\, Z_{ij}e_{j}+ {\rm Im}\, Z_{ij}I_{B} e_{j})$,  where $I_{B}$ is the complex structure on $B$.  Then $$\sum d_{i}^{-1}\delta_{i}\wedge \xi_{i}=\sum ({\rm Im}\,Z)_{ij}{\rm Re}\,dy_{i}\wedge {\rm Im}\, dy_{j}=\omega.$$  Then  \cite[Theorem 3.4(a)]{Fr} shows that $\omega$ is a special K\"ahler metric on $N_0$.
\end{proof}
We now prove the last two statements in Theorem \ref{theorem2}.
Since $\delta_{i}$ and $\xi_{i}$ are closed 1-forms, there are flat Darboux  coordinates $v_{1}, \cdots, v_{2n}$  such that $\delta_{i}=dv_{i}$ and $\xi_{i}=dv_{n+i}$.  Thanks to \cite[Proposition 1.24]{Fr},
 the corresponding Riemannian metric $g_{\omega}$ is the Hessian of a smooth function  $\mathcal{G} $ in the coordinates $v_{i}$, i.e.,  $$g_{\omega}=\sum \frac{\partial^{2}\mathcal{G}}{\partial v_{i}\partial v_{j} }dv_{i}\otimes dv_{j}. $$  Then we have $$
 \sqrt{\det (g_{\omega, ij})}dv_{1} \wedge \cdots \wedge dv_{2n} =   \omega^{n}  =  \left(\prod_i d_{i}^{-1}\right) dv_{1} \wedge \cdots \wedge dv_{2n}.$$
Therefore $g_{\omega}$ is a  Monge-Amp\`ere metric. Finally, thanks to \cite[Remark 3.5]{Fr}, the transition functions of two such coordinates $v_{i}$ and $\tilde{v}_{i}$ satisfy that $$ (v_{1},  \cdots , v_{2n} )= P (\tilde{v}_{1},   \cdots , \tilde{v}_{2n} )+ (b_{1},   \cdots , b_{2n} ),$$ where $P\in Sp(2n, \mathbb{Z})$ and $b_{i}\in \mathbb{R}$.  Thus these coordinates $v_{i}$ define an integral affine structure on $N_{0}$.
This completes the proof of Theorem \ref{theorem2}.


\begin{thebibliography}{99}

\bibitem{DB} P.S. Aspinwall, T. Bridgeland, A. Craw, M.R. Douglas, M. Gross, A. Kapustin, G.W. Moore, G. Segal, B. Szendr\"{o}i, P.M.H. Wilson, {\em Dirichlet branes and mirror symmetry}, Clay Mathematics Institute, Cambridge, MA, 2009.
\bibitem{CC1} J. Cheeger, T. Colding, {\em On the structure of space with Ricci curvature bounded below I}, J. Differential Geom. {\bf 46} (1997), 406--480.
\bibitem{CC2} J. Cheeger, T. Colding, {\em On the structure of space with Ricci curvature bounded below II}, J. Differential Geom. {\bf 52} (1999), 13--35.\
\bibitem{CY} S.Y. Cheng, S.-T. Yau, {\em The real Monge-Amp\`ere equation and affine flat structures}, in {\em Proceedings of the 1980 Beijing Symposium on Differential Geometry and Differential Equations, Vol. 1, 2, 3 (Beijing, 1980)}, 339--370, Science Press, Beijing, 1982.
\bibitem{Co} V. Cort\'es, {\em Special Kaehler manifolds: a survey}, in {\em Proceedings of the 21st Winter School "Geometry and Physics'' (Srn\'i, 2001)},
Rend. Circ. Mat. Palermo (2) Suppl. {\bf 2002}, no. 69, 11--18.
  \bibitem{DS} S. Donaldson, S. Sun,  {\em Gromov-Hausdorff limits of K\"ahler manifolds and algebraic geometry}, arXiv:1206.2609.
\bibitem{Fr} D. Freed,  {\em Special K\"ahler Manifolds}, \rm  Commun. Math. Phys. {\bf 203} (1999), 31--52.
\bibitem{GL} D. Greb, C. Lehn, {\em Base manifolds for Lagrangian fibrations on hyperk\"ahler manifolds}, arXiv:1303.3919.
\bibitem{Gr} P. Griffiths, ed., {\em Topics in transcendental algebraic geometry,}  Annals of Mathematics Studies, 106. Princeton University Press, Princeton, NJ, 1984.
\bibitem{Gr1} M. Gromov, {\em Metric structures for Riemannian and non-Riemannian spaces}, Birkh\"{a}user 1999.
\bibitem{Gro} M. Gross, {\em Mirror Symmetry and the Strominger-Yau-Zaslow conjecture}, arXiv:1212.4220.
\bibitem{GTZ} M. Gross, V. Tosatti, Y. Zhang, {\em  Collapsing of Abelian Fibered Calabi-Yau Manifolds}, Duke Math. J. {\bf 162} (2013), no. 3, 517--551.
\bibitem{GW} M. Gross, P.M.H. Wilson, \emph{Large complex structure limits of $K3$ surfaces}, J. Differential Geom. {\bf 55} (2000), no. 3, 475--546.
\bibitem{Hi1} N. Hitchin, {\em The moduli space of special Lagrangian submanifolds}, Ann. Scuola Norm. Sup. Pisa Cl. Sci. (4) {\bf 25} (1997), no. 3-4, 503--515.
\bibitem{Hi2} N. Hitchin, {\em The moduli space of complex Lagrangian submanifolds},  Asian J. Math. {\bf 3} (1999), no. 1, 77--91.
\bibitem{Hw} J.-M. Hwang, \emph{Base manifolds for fibrations of projective irreducible symplectic manifolds}, Invent. Math. {\bf 174} (2008), 625--644.
\bibitem{KKMS} G. Kempf, F. Knudsen, D. Mumford, B. Saint-Donat, {\em Toroidal embeddings. I.} Lecture Notes in Mathematics, Vol. 339. Springer-Verlag, Berlin-New York, 1973. viii+209 pp.
\bibitem{KS} M. Kontsevich, Y. Soibelman, {\em  Homological mirror symmetry and torus fibrations}, in {\em Symplectic geometry and mirror symmetry,} 203--263, World Sci. Publishing 2001.
\bibitem{La} A. Landman, {\em On the Picard-Lefschetz transformation for algebraic manifolds acquiring general singularities,} Trans. Amer. Math. Soc. {\bf 181} (1973), 89--126.
\bibitem{LYZ} J. Loftin, S.-T. Yau, E. Zaslow, {\em Affine manifolds, SYZ geometry and the ``Y'' vertex}, J. Differential Geom. {\bf 71} (2005), no. 1, 129--158.
\bibitem{Lu} Z. Lu {\em A note on special K\"ahler manifolds},  Math. Ann. {\bf 313} (1999), no. 4, 711--713.
\bibitem{Man} Y.I. Manin {\em  Moduli, motives, mirrors}, in {\em European Congress of Mathematics, Vol. I (Barcelona, 2000)}, 53--73, Progr. Math., 201, Birkh\"auser, Basel, 2001.
\bibitem{Ma} D. Matsushita, \emph{On fibre space structures of a projective irreducible symplectic manifold,} Topology {\bf 38} (1999), 79--83.
\bibitem{Oda} T. Oda, \emph{Convex bodies and algebraic geometry. An introduction to the theory of toric varieties.} Ergebnisse der Mathematik und ihrer Grenzgebiete (3), {\bf 15}, Springer-Verlag, Berlin, 1988.
\bibitem{RZ} X. Rong, Y. Zhang, {\em Continuity of extremal transitions and flops for Calabi-Yau manifolds}, J.  Differential Geom. {\bf 89} (2011), no. 2, 233--269.
\bibitem{RZ2} X. Rong, Y. Zhang, {\em Degenerations of Ricci-Flat Calabi-Yau Manifolds},  to appear in Commun. Contemp. Math. 2013.
\bibitem{RuZ} W.D.  Ruan, Y. Zhang, {\em  Convergence of Calabi-Yau manifolds},  Adv. Math.  {\bf 228} (3) (2011) 1543--1589.
\bibitem{ST1}  J. Song, G. Tian,  \emph{ The K\"ahler-Ricci flow on surfaces of positive Kodaira dimension}, Invent. Math. {\bf 170} (2007), 609--653.
\bibitem{ST2} J. Song, G. Tian,   \emph{  Canonical measures and the K\"ahler-Ricci flow}, J. Amer. Math. Soc. {\bf 25} (2012), no. 2, 303--353.
\bibitem{Steen} J. Steenbrink, {\em Limits of Hodge structures},
Invent.\ Math.\ {\bf 31}, (1975/76), 229--257.
\bibitem{St} A. Strominger, \emph{Special geometry}, Comm. Math. Phys. {\bf 133} (1990), no. 1, 163--180.
\bibitem{SYZ} A. Strominger, S.-T. Yau, E. Zaslow, {\em Mirror symmetry is $T$-duality}, Nuclear Phys. B {\bf 479} (1996), no. 1-2, 243--259.
\bibitem{Tian}  G. Tian,  {\em K\"ahler-Einstein metrics on algebraic manifolds}, in Metric and Differential Geometry, Progress in Mathematics, Vol. {\bf 297} (Birkh\"auser, 2012), pp. 119--159.
\bibitem{To0} V. Tosatti, \emph{Limits of Calabi-Yau metrics when the K\"ahler class degenerates}, J. Eur. Math. Soc. (JEMS) {\bf 11} (2009), no.4, 755-776.
\bibitem{To1} V. Tosatti, {\em  Adiabatic limits of Ricci-flat K\"ahler metrics}, J. Differential Geom. {\bf 84} (2010), no. 2, 427--453.
\bibitem{To2} V. Tosatti, {\em Degenerations of Calabi-Yau metrics}, in {\em Geometry and Physics in Cracow}, Acta Phys. Polon. B Proc. Suppl. {\bf 4} (2011), no.3, 495--505.
\bibitem{To3} V. Tosatti, {\em Calabi-Yau manifolds and their degenerations}, Ann. N.Y. Acad. Sci. {\bf 1260} (2012), 8--13.
\bibitem{V83} E. Viehweg, \emph{Weak positivity and the additivity of the Kodaira dimension for certain fibre spaces}, in Advances Studies in Pure Mathematics {\bf 1}, 1983, Algebraic Varieties and Analytic Varieties, 329--353.
\bibitem{Wa} C.-L. Wang, {\em On the incompleteness of the Weil-Petersson metric along degenerations of Calabi-Yau manifolds}, Math. Res. Lett. {\bf 4} (1997), no. 1, 157--171.
\bibitem{Ya} S.-T. Yau, \emph{On the Ricci curvature of a compact K\"ahler manifold and the complex Monge-Amp\`ere equation, I}, Comm. Pure Appl. Math. {\bf 31} (1978), 339--411.
\bibitem{Z} Y. Zhang, {\em Convergence of K\"ahler manifolds and calibrated fibrations}, PhD thesis, Nankai Institute of Mathematics, 2006.
 \end{thebibliography}
\end{document}